\documentclass[12pt,a4paper]{article}
\usepackage{amssymb}
\usepackage{amsmath}
\usepackage{amsfonts}

\newtheorem{theorem}{Theorem}

\newtheorem{corollary}[theorem]{Corollary}

\newtheorem{definition}[theorem]{Definition}
\newtheorem{example}[theorem]{Example}

\newtheorem{proposition}[theorem]{Proposition}
\newtheorem{remark}[theorem]{Remark}

\newenvironment{proof}[1][Proof]{\noindent\textbf{#1.} }{\ \rule{0.5em}{0.5em}}
\input{tcilatex}
\begin{document}

\title{Standard transmutation operators for the one dimensional Schr\"{o}%
dinger operator with a locally integrable potential}
\author{Hugo M. Campos \\
{\small School of Mathematics, Yachay Tech, Yachay City of Knowledge, }\\
{\small 100119, Urcuqui, Ecuador}\\
{\small email: hfernandes@yachaytech.edu.ec}}
\maketitle

\begin{abstract}
We study a special class of operators $T$ satisfying the transmutation
relation 
\begin{equation*}
(\frac{d^{2}}{dx^{2}}-q)Tu=T\frac{d^{2}}{dx^{2}}u
\end{equation*}%
in the sense of distributions, where $q$ is a locally integrable function,
and $u$ belongs to a suitable space of distributions depending on the
smoothness properties of $q$.

A method which allows one to construct a fundamental set of transmutation
operators of this class in terms of a single particular transmutation
operator is presented. Moreover, following \cite{Marchenco}, we show that a
particular transmutation operator can be represented as a Volterra integral
operator of the second kind.

We study the boundedness and invertibility properties of the transmutation
operators, and use these to obtain a representation for the general
distributional solution of the equation $\frac{d^{2}u}{dx^{2}}-qu=\lambda u$%
, $\lambda \in \mathbb{C}$, in terms of the general solution of the same
equation with $\lambda =0.\bigskip $

\end{abstract}

\section{Introduction}

If $L$ and $M$ are differential operators, then an operator $T$ satisfying
the relation $LT=TM$, in a suitable functional space, is called a
transmutation operator. Transmutation operators are a useful tool in the
study of differential operators with variable coefficients, see \cite{Beghar}%
, \cite{Carrol}, \cite{LevitanInverse}, \cite{Lions}, and \cite{Marchenco}
for their classical theory.

In the present paper we deal mainly with transmutations corresponding to the
pair $L=\frac{d^{2}}{dx^{2}}-q(x)$ and $M=\frac{d^{2}}{dx^{2}}$. It is well
known that if $q\in C\left[ -a,a\right] $, $a>0$, then a transmutation
operator for such pair on $C^{2}\left[ -a,a\right] $ can be represented in
the form%
\begin{equation}
\mathbf{T}u(x)=u(x)+\int_{-x}^{x}K(x,t)u(t)dt,  \label{VolterraIntro}
\end{equation}%
where $K(x,t)$ is the unique solution of a certain Goursat problem \cite%
{Marchenco}.

Recent work \cite{CKT} shows an interesting connection between the
transmutation operators, and the concept of $L$-base introduced in \cite%
{Fage}. There, a parametrized family of operators $\left\{ T_{c}\right\} $
was introduced possessing the following properties: $i)$ for each $c\in 
\mathbb{C}$, $T_{c}$ is a Volterra integral operator and $T_{0}=\mathbf{T}$; 
$ii)$ if $q\in C^{1}[-a,a]$, then $T_{c}$ is a transmutation operator on $%
C^{2}\left[ -a,a\right]$ for suitable values of $c$; $iii)$ the family of
functions $\left\{ \varphi _{k}\right\} $, where $\varphi _{k}=T_{c}[x^{k}],$
$k=0,1,2,..,$ is an $L$-base.

It is worth mentioning that the $L$-base $\left\{ \varphi _{k}\right\} $ is
the main ingredient to obtain the so called spectral parameter power series
(SPPS) representation, a useful representation of the general solution of
the Sturm--Liouville equation \cite{KravSPPS}, \cite{KP}, and it also arises
in Bers's theory of pseudoanalytic functions \cite{Bers Book}, \cite{KBook}
in the study of a special kind of Vekua equations. These connections have
led to new applications of the transmutation operators to the study of
different problems arising in mathematical physics \cite{CMK}, \cite{CKT}, 
\cite{KKTmod}, \cite{KKTT}, \cite{KST}, \cite{KTacurate}.

Regarding these applications, it is desirable to have the transmutation
property of $T_{c}$ under the most general scenario, which, to the best of
our knowledge, is valid for a continuous potential $q$ and for all
parameters $c\in \mathbb{C}$ (see Remark \ref{Trans_Relax} for further
details).

The first goal of this paper is the study of transmutations for the Schr\"{o}%
dinger operator with integrable coefficients. Our approach is not based on
merely extending existing results from continuous coefficients, but instead
we make our study in the spirit of standard $L$-bases (see Definition \ref%
{standard}). We will consider the class of transmutation operators, here
named standard, which have the property of mapping the nonnegative integer
powers of the independent variable into an s-$L$-base. This class includes
the transmutation operators having the form of a Volterra integral of the
second kind, and others.

This paper is organized as follows. In Section 2 we fix some notation and we
give some auxiliary results needed throughout the paper.

In Section 3, based on the properties of the s-$L$-bases, we arrive at
several results concerning the s-transmutation operators, and we establish
the first main result of the paper, Theorem \ref{Thmapping}, that if a
bounded operator $T$ on $L^{1}[a,b]$ maps the powers of the independent
variable into an s-$L$-base, then $T$ is a transmutation operator on $%
W^{2,1}[a,b]$. This result represents another proof for the transmutation
property of the operator $T_{c}$ and also suggests that some properties or
relations on the s-transmutations might be studied first on smaller spaces,
like the linear space of polynomials $\mathcal{P}(\mathbb{R})$, and then
extended to suitable Banach spaces.

In Section 4 we present a fairly simple method for constructing new
transmutation operators on $\mathcal{P}(\mathbb{R)}$ when one single
s-transmutation is known (Theorem \ref{TheoTrans}) and we also provide
explicit formulas relating any two s-transmutations on $\mathcal{P}(\mathbb{%
R)}$, allowing the construction of one of them in terms of the other
(Proposition \ref{PropTphiTpsi}).

From the viewpoint of applications, it is desirable to have the
transmutation property (\ref{TransOp}) on larger spaces than $\mathcal{P}(%
\mathbb{R)}$. Section 5 is devoted to the study of s-transmutations on
Sobolev spaces. Extending the results of \cite{Marchenco}, we show that a
particular s-transmutation operator on $W^{2,1}[-a,a]$ can be represented in
the form (\ref{VolterraIntro}), where $K(x,t)$ is a weak solution of a
Goursat problem (Theorem \ref{Prop_exist_transm}). Combining this result
with those of Section 4, we arrive at the main result of the paper, Theorem %
\ref{Fulltransmu}, an explicit representation of the general s-transmutation
operator on Sobolev space. As an application, we get the SPPS representation
for the general solution of the equation $\frac{d^{2}u}{dx^{2}}-qu=\lambda u$
in a slightly more general setting than was known previously (Corollary \ref%
{CorSPPS}).

Even in the case of continuous coefficients, the transmutation relation (\ref%
{TransOp}) has been studied mainly in the space of twice differentiable
functions, and assuming the existence of all the involved derivatives in the
classical sense. However, as always, for the best understanding of
differential operators, the framework of distributions is the most suitable.

The second goal of the paper is to study the s-transmutations on spaces of
distributions, which is done in Section 6. Assuming that the potential of
the Schr\"{o}dinger operator is locally integrable, we show that each
s-transmutation operator admits a continuous extension as a linear operator
and as a transmutation operator to a suitable space of distributions,
depending on the smoothness of the potential $q$. Also, this procedure leads
to the construction of all the distributional s-transmutation operators.

Lastly, in Section 7 we provide some conclusions and generalizations.

\section{Preliminaries\label{SectionPre}}

First, following \cite{Neto}, \cite{Brezis}, and \cite{Treves}, let us
introduce the spaces we are going work with. Let $\mathcal{(}a,b)\subset 
\mathbb{R}$ be a bounded open interval and denote by $F(a,b)$ any one of the
following classical spaces: $L^{p}(a,b)$, $C^{k}(a,b)$, or $C^{\infty }(a,b)$%
, for $k\in \mathbb{N}_{0}=\left\{ 0,1,2..\right\} $ and $1\leq p\leq \infty 
$. Note that $L^{p}(a,b)$ and $C^{k}(K)$ are Banach spaces with the
corresponding natural norms whenever $K\subset (a,b)$ is compact, and $%
C^{\infty }(K)$ is a Fr\'{e}chet space when equipped with the topology of
uniform convergence on $K$ in each derivative \cite{Neto}.

Let us denote by $F_{loc}(a,b)$ the subspace of $F(a,b)$ formed by those
functions $\psi \in F(K)$, for each compact set $K\subset (a,b)$. By $%
F_{c}(a,b)$ we denote the set of functions $\psi \in F(a,b)$ such that $%
\func{supp}\psi \subset (a,b)$, here $\func{supp}\psi $ is defined as the
complement of the largest open set on which $\psi $ vanishes a.e. We endow $%
F_{c}(a,b)$ with the inductive limit topology, turning $F_{c}(a,b)$ into a
locally convex topological vector space \cite{Neto}, \cite{Treves}. A
sequence of functions $\psi _{n}\in F_{c}(a,b)$ converges to $\psi \in
F_{c}(a,b)$ with respect to this topology if there exists a compact set $%
K\subset (a,b)$ such that $\func{supp}\psi _{n}\subset K$ for all $n$ and $%
\psi _{n}\rightarrow \psi $ in $F(K)$ as $n\rightarrow +\infty $. Let $%
(F_{c}(a,b))^{\prime }$ be the dual space of $F_{c}(a,b)$, consisting of all
continuous functionals $u:F_{c}(a,b)\rightarrow \mathbb{C}$. The value of a
functional $u$ on $\psi \in $ $F_{c}(a,b)$ is denoted by $\left\langle
u,\psi \right\rangle $ and continuity means that $\left\langle u,\psi
_{n}\right\rangle \rightarrow \left\langle u,\psi \right\rangle $ provided
that $\psi _{n}\rightarrow \psi $ in $F_{c}(a,b)$ as $n\rightarrow \infty $.
Convergence in $(F_{c}(a,b))^{\prime }$ is defined to be pointwise
convergence. Obviously, for a smaller set of $F_{c}(a,b)$, the set $%
(F_{c}(a,b))^{\prime }$ will be bigger.

The space $C_{c}^{\infty }(a,b)$, called the space of test functions, plays
an important role in the theory of distributions. As defined above, $\phi
_{n}\rightarrow \phi $ in $C_{c}^{\infty }(a,b)$ if there exists a compact
set $K\subset (a,b)$ such that $\func{supp}\phi _{n}\subset K$ for all $n$
and $\frac{d^{m}\phi _{n}}{dx^{m}}\rightarrow \frac{d^{m}\phi }{dx^{m}}$
uniformly on $K$ as $n\rightarrow \infty $ for each $m\in \mathbb{N}_{0}$. A
distribution is an element from the dual of $C_{c}^{\infty }(a,b)$, which
from now on we denote by $\mathcal{D}^{\prime }(a,b)$, i.e., $\mathcal{D}%
^{\prime }(a,b)\equiv (C_{c}^{\infty }(a,b))^{\prime }$. A distribution $u$
is called regular if there exists an $f\in L_{loc}^{1}(a,b)$ such that $%
\left\langle u,\phi \right\rangle =\tint_{a}^{b}f(x)\phi (x)dx$ for all $%
\phi \in C_{c}^{\infty }(a,b)$. We usually identify a regular distribution $%
u $ with the corresponding function $f$, $u\equiv f$.

\begin{definition}
Let $T$ and $T^{\square }$ be linear operators from $C_{c}^{\infty }(a,b)$
to $L_{loc}^{1}(a,b)$ such that $\int_{a}^{b}(T\psi )\phi
dx=\int_{a}^{b}\psi (T^{\square }\phi )dx$ for all $\psi ,\phi \in
C_{c}^{\infty }(a,b)$. The operator $T^{\square }$, if it exists, is called
the transpose of $T$.
\end{definition}

The most common way to extend linear operations from functions to
distributions is the following.

\begin{proposition}
\label{PropLext}Let $T$ be a linear operator defined on $C_{c}^{\infty
}(a,b) $. If $T^{\square }:C_{c}^{\infty }(a,b)\rightarrow C_{c}^{\infty
}(a,b)$ is continuous, then the operator $T$ can be extended to a continuous
linear operator from $\mathcal{D}^{\prime }(a,b)$ to itself by the formula 
\begin{equation*}
\left\langle Tu,\phi \right\rangle :=\left\langle u,T^{\square }\phi
\right\rangle \quad ,u\in \mathcal{D}^{\prime }(a,b),\phi \in C_{c}^{\infty
}(a,b).
\end{equation*}
\end{proposition}

\begin{proof}
This is a special case of \cite[Th. 3.8]{Grubb}.
\end{proof}

\bigskip

As $\frac{d}{dx}^{\square }=-\frac{d}{dx}$ is continuous on $C_{c}^{\infty
}(a,b)$, the derivative of $u\in \mathcal{D}^{\prime }(a,b)$ is defined by
employing the previous proposition: $\left\langle u^{\prime },\phi
\right\rangle :=-\left\langle u,\phi ^{\prime }\right\rangle $. It is again
an element of $\mathcal{D}^{\prime }(a,b)$.

The Sobolev space $W^{k,p}(a,b)$, $k\in \mathbb{N}$, $1\leq p\leq \infty $
consists of all functions $f\in L^{p}(a,b)$ having distributional
derivatives $f^{(n)}\in L^{p}(a,b)$ up to order $n\leq k$.

Denote the set of complex-valued absolutely continuous functions on $[a,b]$
by $AC[a,b]$. We recall that a function $V$ belongs to $AC[a,b]$ iff there
exists $v\in L^{1}(a,b)$ and $x_{0}\in \lbrack a,b]$ such that $%
V(x)=\int_{x_{0}}^{x}v(t)dt+V(x_{0})$,$\forall \,\ x\in \lbrack a,b]$ (see,
e.g., \cite{Neto}, \cite{Kolmogorov}). Thus, if $V\in AC[a,b]$, then $V\in
C[a,b]$ and $V$ is a.e. differentiable with $V^{\prime }(x)=v(x)$ on $%
\mathcal{(}a,b)$. Moreover, the usual derivative of a function from $AC[a,b]$
coincides a.e. with its distributional derivative \cite{Kolmogorov}, thus $%
AC[a,b]\subseteq W^{1,1}(a,b)$. In fact, the linear spaces $W^{1,1}(a,b)$
and $AC[a,b]$ coincide (see \cite{Brezis}, Theorem 8.2) in the following
sense: $u\in W^{1,1}(a,b)$ iff there exists $V\in AC[a,b]$ such that $u=V$
a.e. on $\mathcal{(}a,b)$.

Consider the differential equation%
\begin{equation}
u^{\prime \prime }+qu=h,  \label{SONE}
\end{equation}%
in the sense of distributions, where $q,h\in L^{1}(a,b)$ and the
distribution $qu$ is defined by $\left\langle qu,\varphi \right\rangle
:=\left\langle u,q\varphi \right\rangle $. Note that if $q\notin C^{\infty
}(a,b)$, then it might happen that $u\in \mathcal{D}^{\prime }(a,b)$ but $%
qu\notin \mathcal{D}^{\prime }(a,b)$. To overcome this situation, we first
will look for solutions of (\ref{SONE}) in $W^{1,1}(a,b)$, which are called
weak solutions.

\begin{proposition}
\label{Prop_weak}Let $q$ and $h$ be functions in $L^{1}(a,b)$. Then a
function $u\in W^{1,1}(a,b)$ is a weak solution of (\ref{SONE}) iff $u\in
W^{2,1}(a,b)$ and (\ref{SONE}) is satisfied a.e. on $\mathcal{(}a,b).$
\end{proposition}

\begin{proof}
See \cite{BCK}.
\end{proof}

The last result reveals the density of polynomials in $W^{2,1}(a,b)$.

\begin{proposition}
\label{pndenseDl}Let $u\in W^{2,1}(a,b)$. Then there exists a sequence of
polynomials $P_{n}$ such that $P_{n}\rightarrow u,$ $P_{n}^{\prime
}\rightarrow u^{\prime }$ uniformly on $[a,b]$ and $P_{n}^{\prime \prime
}\rightarrow u^{\prime \prime }$ in $L^{1}(a,b)$, as $n\rightarrow \infty $.
\end{proposition}

\begin{proof}
Let $u\in W^{2,1}(a,b)$. Then $u^{\prime \prime }\in L^{1}(a,b)$, and
because polynomials are dense in $L^{1}(a,b)$ there exists a sequence of
polynomials $Q_{n}$ such that%
\begin{equation}
Q_{n}\rightarrow u^{\prime \prime }\text{ in }L^{1}(a,b)\text{, as }%
n\rightarrow \infty .  \label{polconv}
\end{equation}%
Integrating (\ref{polconv}) twice we see that the sequence of polynomials $%
P_{n}(x):=u(0)+u^{\prime }(0)x+\int_{0}^{x}\int_{0}^{t}Q_{n}(s)\,ds\,dt$
fulfills the conclusion of the proposition.
\end{proof}

\section{s-transmutations and $L$-bases. First results}

Consider the one-dimensional Schr\"{o}dinger operator%
\begin{equation}
L=\frac{d^{2}}{dx^{2}}-q(x),  \label{SchOp}
\end{equation}%
with domain of definition $D_{L}=W^{2,1}(a,b)$, where $q\in L^{1}(a,b).$

\begin{definition}
A linear map $T$ is called a transmutation operator on $\chi \subseteq D_{L}$
if for every $u\in \chi $ we have $u^{\prime \prime }\in D_{T}$, $Tu\in $ $%
D_{L}$ and the following equality holds%
\begin{equation}
\left( \frac{d^{2}}{dx^{2}}-q\right) Tu=T\frac{d^{2}}{dx^{2}}u.
\label{TransOp}
\end{equation}
\end{definition}

Let $\mathcal{P}(\mathbb{R})$ denote the linear space of polynomials. We
will be mainly interested in operators satisfying the transmutation property
(\ref{TransOp}) on sets $\chi \supseteq \mathcal{P}(\mathbb{R})$. Let $T$ be
a transmutation operator and define $\varphi _{k}:=T\left[ x^{k}\right] $, $%
k\in \mathbb{N}_{0}$. From the transmutation property (\ref{TransOp}), we
obtain%
\begin{equation*}
\left( \frac{d^{2}}{dx^{2}}-q\right) \varphi _{k}=\left( \frac{d^{2}}{dx^{2}}%
-q\right) Tx^{k}=T\frac{d^{2}}{dx^{2}}x^{k}.
\end{equation*}%
Thus,%
\begin{equation}
\left( \frac{d^{2}}{dx^{2}}-q\right) \varphi _{k}=0,\text{\quad }k=0,1;
\label{L-base0,1}
\end{equation}%
and%
\begin{equation}
\left( \frac{d^{2}}{dx^{2}}-q\right) \varphi _{k}=k(k-1)\varphi _{k-2},\text{%
\quad }k\geq 2.  \label{L-baseplus2}
\end{equation}%
As in \cite{Fage}, we define the following:

\begin{definition}
\label{DefLBase} A system of functions $\left\{ \varphi _{k}\right\} $, $%
k\in \mathbb{N}_{0}$, satisfying (\ref{L-base0,1}) and (\ref{L-baseplus2})
in the sense of weak solutions, is called an $L$-base.
\end{definition}

We have shown that if $T$ is a transmutation operator on $\mathcal{P}(%
\mathbb{R})$, then $\varphi _{k}:=T\left[ x^{k}\right] $ is an $L$-base.
Moreover, given an $L$-base $\left\{ \varphi _{k}\right\} $, there exists a
transmutation operator $T_{\varphi }$ on $\mathcal{P}(\mathbb{R})$ such that 
$\varphi _{k}=T_{\varphi }\left[ x^{k}\right] $ for all $k\in \mathbb{N}_{0} 
$. To see this, define $T_{\varphi }$ on powers to be $T_{\varphi }\left[
x^{k}\right] :=\varphi _{k}$ and extend it to $\mathcal{P}(\mathbb{R}) $ by
linearity. It follows from Proposition \ref{Prop_weak} that each member of
an $L$-base is a function from the space $W^{2,1}(a,b)$, hence $T_{\varphi }$
maps $\mathcal{P}(\mathbb{R})$ into $W^{2,1}(a,b)$. Also, if $%
P=\sum\limits_{k=0}^{m}\alpha _{k}x^{k}$, from (\ref{L-base0,1}) and (\ref%
{L-baseplus2}) we get 
\begin{equation*}
(\frac{d^{2}}{dx^{2}}-q)T_{\varphi }P=\sum\limits_{k=0}^{m}\alpha _{k}(\frac{%
d^{2}}{dx^{2}}-q)\varphi _{k}=\sum\limits_{k=2}^{m}\alpha _{k}k(k-1)\varphi
_{k-2}=T_{\varphi }P^{\prime \prime }.
\end{equation*}%
This proves the existence of a transmutation operator on $\mathcal{P}(%
\mathbb{R})$ satisfying $T_{\varphi }\left[ x^{k}\right] =\varphi _{k}$.

However, regarding possible applications, it is important to establish the
existence of transmutations on larger domains than $\mathcal{P}(\mathbb{R})$%
, for example on $W^{2,1}(a,b),$ and investigate when they are bounded or
invertible. Looking for the existence of such operators, we will work with a
special class of $L$-bases, which we introduce below. We will assume from
now on, without loss of generality, that $0\in \lbrack a,b]$.

\begin{definition}
\label{standard}A standard $L$-base (or, simply, an s-$L$-base) is an $L$%
-base $\left\{ \varphi _{k}\right\} $ such that $\varphi _{k}(0)=\varphi
_{k}^{\prime }(0)=0$ for all $k\geq 2$. A transmutation operator is called
standard if $\varphi _{k}:=T\left[ x^{k}\right] $ is an s-$L$-base.
\end{definition}

\begin{example}
The system of powers $\left\{ x^{k}\right\} $, $k\in \mathbb{N}_{0},$ is an
s-$L$-base corresponding to the operator $L=\frac{d^{2}}{dx^{2}}$. The shift
operator $Eu(x)=u(x-x_{0})$, $x_{0}\neq 0$ is a transmutation operator (on $%
W^{2,1}(\mathbb{R})$) corresponding to $L=\frac{d^{2}}{dx^{2}}$ but it is
not an s-transmutation operator.
\end{example}

\begin{remark}
\label{Remark s-l-base}We stress that an s-$L$-base $\left\{ \varphi
_{k}\right\} $ is completely determined by its first two elements $\varphi
_{0}$ and $\varphi _{1}$. For instance, from Proposition \ref{Prop_weak} and
by the variation of parameters formula, we see that the unique weak solution
of (\ref{L-baseplus2}) satisfying $\varphi _{k}(0)=\varphi _{k}^{\prime
}(0)=0$ is given by%
\begin{equation}
\varphi _{k}(x)=k(k-1)\dint\limits_{0}^{x}G(x,s)\varphi _{k-2}(s)ds,\text{%
\quad }k\geq 2,  \label{slbaseconstk>2}
\end{equation}%
where%
\begin{equation}
G(x,s)=\dfrac{\psi _{0}(s)\psi _{1}(x)-\psi _{0}(x)\psi _{1}(s)}{W(\psi
_{0},\psi _{1})},  \label{GreenFu}
\end{equation}%
$W(\psi _{0},\psi _{1})=\psi _{0}(x)\psi _{1}^{\prime }(x)-\psi _{0}^{\prime
}(x)\psi _{1}(x)$, and $\psi _{0},\psi _{1}$ are any two linearly
independent weak solutions of (\ref{L-base0,1}). Note that $G(x,s)$ does not
depend on the choice of $\psi _{0}$ and $\psi _{1}$.
\end{remark}

\begin{theorem}
\label{Thmapping}Suppose $\left\{ \varphi _{k}\right\} $ is an s-$L$-base
and suppose $T:L^{1}(a,b)\rightarrow L^{1}(a,b)$ is a bounded linear
operator such that $T[x^{k}]=\varphi _{k}$, $k\in \mathbb{N}_{0}$. Then $T$
is an s-transmutation operator on $W^{2,1}(a,b)$.
\end{theorem}

\begin{proof}
Let us prove that%
\begin{equation}
T[u]=u(0)\varphi _{0}(x)+u^{\prime }(0)\varphi _{1}(x)+\int_{0}^{x}G(x,s)T 
\left[ u^{\prime \prime }\right] (s)ds  \label{TransOpGrenn}
\end{equation}%
for any $u\in W^{2,1}(a,b)$, where $G(x,s)$ is given by (\ref{GreenFu}).
Then (\ref{TransOp}) will follow from the application of $L=\frac{d^{2}}{%
dx^{2}}+q(x)$ to both sides of (\ref{TransOpGrenn}).

As $T$ is linear and $T[x^{k}]=\varphi _{k}$ is an s-$L$-base, the
application of $T$ to $P=\sum\limits_{k=0}^{m}\alpha _{k}x^{k}$ gives%
\begin{eqnarray*}
T\left[ P\right] &=&\dsum\limits_{k=0}^{m}\alpha
_{k}T_{h}[x^{k}]=\dsum\limits_{k=0}^{m}\alpha _{k}\varphi _{k}(x) \\
&=&\alpha _{0}\varphi _{0}(x)+\alpha _{1}\varphi
_{1}(x)+\int_{0}^{x}G(x,s)\dsum\limits_{k=2}^{m}k(k-1)\alpha _{k}\varphi
_{k-2}(s)ds \\
&=&P(0)\varphi _{0}(x)+P^{\prime }(0)\varphi _{1}(x)+\int_{0}^{x}G(x,s)T 
\left[ P^{\prime \prime }\right] (s)ds,
\end{eqnarray*}%
where we have used (\ref{slbaseconstk>2}). According to Proposition \ref%
{pndenseDl}, given $u\in W^{1,2}(a,b)$ there exists a sequence of
polynomials $P_{n}$ such that $P_{n}\rightarrow u$, $P_{n}^{\prime
}\rightarrow u^{\prime }$ uniformly on $[a,b]$ and $P_{n}^{\prime \prime
}\rightarrow u^{\prime \prime }$ in $L_{1}(a,b)$, as $n\rightarrow \infty $.
Then (\ref{TransOpGrenn}) follows from a simple limiting procedure using the
continuity of $T$ and the above relations,%
\begin{eqnarray*}
T\left[ u\right] &=&\underset{n\rightarrow \infty }{\lim }T\left[ P_{n}%
\right] =\underset{n\rightarrow \infty }{\lim }\left[ P_{n}(0)\varphi
_{0}(x)+P_{n}^{\prime }(0)\varphi _{0}(x)+\int_{0}^{x}G(x,s)T\left[
P_{n}^{\prime \prime }\right] (s)ds\right] \\
&=&u(0)\varphi _{0}(x)+u^{\prime }(0)\varphi _{1}(x)+\int_{0}^{x}G(x,s)T 
\left[ u^{\prime \prime }\right] (s)ds.
\end{eqnarray*}
\end{proof}

\begin{remark}
\label{ReThmapping} It is easy to see that the above theorem is still true
if we replace the boundedness of $T$ on $L^{1}(a,b)$ by the boundedness from 
$W^{1,1}(a,b)$ to $L^{1}(a,b)$. However, in this case the transmutation
property (\ref{TransOp}) holds on $W^{3,1}(a,b)$, instead of on $%
W^{2,1}(a,b) $.
\end{remark}

The following proposition establishes conditions under which the limit of a
sequence of s-transmutation operators is an s-transmutation operator as well.

\begin{proposition}
\label{PropConvTn}Let $q_{n}$ be a convergent sequence in $L^{1}(a,b)$ with
limit $q$. Let $T$ be a bounded operator on $L^{1}(a,b)$ and $T_{n}$ be a
sequence of s-transmutation operators corresponding to $L_{n}=\frac{d^{2}}{%
dx^{2}}+q_{n}$. If for each $k=0,1,2..$ 
\begin{equation*}
T_{n}\left[ x^{k}\right] \rightarrow T\left[ x^{k}\right] \text{ uniformly
on }[a,b]\text{, as }n\rightarrow \infty ,
\end{equation*}%
then $T$ is an s-transmutation operator on $W^{2,1}(a,b)$ corresponding to $%
L=\frac{d^{2}}{dx^{2}}+q$.
\end{proposition}

\begin{proof}
Write $\varphi _{k,n}(x)=T_{n}\left[ x^{k}\right] $ and $\varphi _{k}(x)=T%
\left[ x^{k}\right] $. To prove the theorem, it suffices to show that $%
\left\{ \varphi _{k}\right\} $ is an s-$L$-base. Indeed, using the
boundedness of $T$ together with the mapping property $T\left[ x^{k}\right]
=\varphi _{k}(x)$, we conclude from Theorem \ref{Thmapping} that $T$ is an
s-transmutation operator on $W^{2,1}(a,b).$

Let us start by proving that $L\left[ \varphi _{0}\right] =\varphi
_{0}^{\prime \prime }(x)-q(x)\varphi _{0}(x)=0$. Consider the functions $%
r_{n}(x)$ and $r(x)$ defined by%
\begin{equation}
r_{n}(x)=\varphi _{0,n}(x)-\int_{0}^{x}(x-t)q_{n}(t)\varphi _{0,n}(t)dt,
\label{intecqn}
\end{equation}%
\begin{equation}
r(x)=\varphi _{0}(x)-\int_{0}^{x}(x-t)q(t)\varphi _{0}(t)dt\text{.}
\label{intecq}
\end{equation}%
It is easy to see from (\ref{intecqn}) that $r^{\prime \prime }(x)=0$ iff $L%
\left[ \varphi _{0}\right] =0$. In this way, as $L_{n}\left[ \varphi _{0,n}%
\right] =\varphi _{0,n}^{\prime \prime }(x)-q_{n}(x)\varphi _{0,n}(x)=0$
from (\ref{intecqn}), we get $r_{n}^{\prime \prime }(x)=0$. Thus $%
r_{n}(x)=c_{1}+c_{2}x$, and taking into account the boundary values at $x=0$%
, we see that 
\begin{equation}
r_{n}(x)=\varphi _{0,n}(0)+\varphi _{0,n}^{^{\prime }}(0)x.  \label{rn}
\end{equation}%
Since, by assumption, $\varphi _{0,n}(x)=T_{n}[1]\rightarrow T[1]=\varphi
_{0}(x)$ uniformly on $[a,b]$ as $n\rightarrow \infty $, we can see from (%
\ref{intecqn}) and (\ref{intecq}) that $r_{n}(x)\rightarrow r(x)$ uniformly
on $[a,b]$. From this and (\ref{rn}) we must have that $\varphi
_{0,n}^{\prime }(0)\rightarrow \varphi _{0}^{\prime }(0)$ as $n\rightarrow
\infty $, and $r(x)=\varphi _{0}(0)+\varphi _{0}^{\prime }(0)x$. Thus $%
r^{\prime \prime }(x)=0$ and from (\ref{intecq}) we conclude that $L\left[
\varphi _{0}\right] =0$. In the same manner we can see that $L\left[ \varphi
_{1}\right] =0$.

Let $G(x,s)$ be the function in (\ref{GreenFu}) where $\psi _{0}$ and $\psi
_{1}$ are linearly independent solutions of $L\left[ \psi \right] =0$ and
denote by $G_{n}(x,s)$ the analogous function corresponding to the operator $%
L_{n}$. From the fact that $q_{n}\rightarrow q$ in $L^{1}(a,b)$ it is easy
to see that $G_{n}(x,s)\rightarrow G(x,s)$ uniformly on $[a,b]\times \lbrack
a,b]$ as $n\rightarrow \infty $. Hence, as $\varphi _{k,n}(x)\rightarrow
\varphi _{k}(x)$ uniformly, as $n\rightarrow \infty $, for $k\geq 2$ we
finally obtain%
\begin{eqnarray*}
\varphi _{k}(x) &=&\underset{n\rightarrow \infty }{\lim }\varphi
_{k,n}(x)=k(k-1)\underset{n\rightarrow \infty }{\lim }\dint%
\limits_{0}^{x}G_{n}(x,s)\varphi _{k-2,n}(s)ds \\
&=&k(k-1)\dint\limits_{0}^{x}G(x,s)\varphi _{k-2}(s)ds.
\end{eqnarray*}%
This shows that $\left\{ \varphi _{k}\right\} $ is an s-$L$-base. The
theorem is proved.
\end{proof}

\section{New s-transmutations\label{Sect4}}

In this section a method for constructing new s-transmutation operators on $%
\mathcal{P}(\mathbb{R})$ in terms of one single s-transmutation operator is
presented.

Let us introduce the operators 
\begin{equation}
P_{\pm }u(x)=\frac{u(x)\pm u(-x)}{2},\quad Au(x)=\int_{0}^{x}u(s)ds\text{,}
\label{Proye}
\end{equation}%
mapping $\mathcal{P}(\mathbb{R})$ into itself. Note that the $P_{\pm }$ are
pairwise projection operators, that is, $P_{+}^{2}=P_{+}$, $P_{-}^{2}=P_{-}$%
, $P_{+}+P_{-}=I$, and $P_{+}P_{-}=P_{-}P_{+}=0$.

\begin{proposition}
\label{PropP+-}The relations $\frac{d}{dx}P_{\pm }=P_{\mp }\frac{d}{dx}$ and 
$AP_{\pm }=P_{\mp }A$ hold on $\mathcal{P}(\mathbb{R})$.
\end{proposition}

\begin{proof}
Straightforward.
\end{proof}

\begin{proposition}
\label{TP+-}Let $T$ be an s-transmutation operator on $\mathcal{P}(\mathbb{R}%
)$. Then

\begin{description}
\item[$i)$] $T[1]=0$ if and only if $T=TP_{-}$ on $\mathcal{P}(\mathbb{R})$;

\item[$ii)$] $T[x]=0$ if and only if $T=TP_{+}$ on $\mathcal{P}(\mathbb{R})$;

\item[$iii)$] $T[1]=0$ and $T[x]=0$ if and only if $T\equiv 0$ on $\mathcal{P%
}(\mathbb{R}).$
\end{description}
\end{proposition}

\begin{proof}
We prove the direct assertions corresponding to $i)$, $ii)$ and $iii)$. The
reciprocal assertions are obvious. First, denote by $\varphi _{k}=T[x^{k}]$
the s-$L$-base corresponding to $T$. Let us prove $i)$. Since $\varphi
_{0}=T[1]=0$, it follows from Remark \ref{Remark s-l-base} that $\varphi
_{k}=T[x^{k}]=0$ for $k\geq 0$ even. Thus it is easy to see that $%
T[x^{k}]=TP_{-}[x^{k}]$ for all $k\geq 0$. From this and from the linearity
of $T$ and $TP_{-}$ we conclude that $T[p(x)]=TP_{-}[p(x)]$ for every
polynomial $p(x)$. Part $ii)$ is proved by analogy. Part $iii)$ is a
consequence of $i)$ and $iii)$. For instance, since $%
T=TI=T(P_{+}+P_{-})=TP_{+}+TP_{-}=2T$, we have $T=0$.
\end{proof}

\begin{corollary}
\label{CoroTP+}Let $T$ and $M$ be s-transmutation operators on $\mathcal{P}(%
\mathbb{R})$ such that $T[1]=M[1]$ and $T[x]=M[x]$. Then $T=M$ on $\mathcal{P%
}(\mathbb{R})$.
\end{corollary}

\begin{proof}
This follows at once from part $iii)$ of the above proposition. Setting $%
S=T-M$, we have that $S$ is an s-transmutation operator and $S[1]=S[x]=0$.
Thus, $S\equiv 0$ on $\mathcal{P}(\mathbb{R})$.
\end{proof}

\begin{proposition}
\label{Propothertrans}Let $T$ be an s-transmutation operator on $\mathcal{P}(%
\mathbb{R})$. Then:

\begin{description}
\item[$i)$] $T\frac{d}{dx}$ is an s-transmutation operator on $\mathcal{P}(%
\mathbb{R})$ if and only if $T[x]=0$, and in this case $T\frac{d}{dx}=T\frac{%
d}{dx}P_{-}$ on $\mathcal{P}(\mathbb{R})$;

\item[$ii)$] $TA$ is an s-transmutation operator on $\mathcal{P}(\mathbb{R})$
if and only if $T[1]=0$, and in this case $TA=TAP_{+}$ on $\mathcal{P}(%
\mathbb{R});$
\end{description}
\end{proposition}

\begin{proof}
Let $T$ be an s-transmutation operator and denote by $\varphi _{k}=T[x^{k}]$
the corresponding s-$L$-base. Let us prove $i)$. The transmutation property
of $T\frac{d}{dx}$ is a straightforward calculation. Analyzing the
corresponding $L$-base $\psi _{k}=T\frac{d}{dx}[x^{k}]=kT\left[ x^{k-1}%
\right] =k\varphi _{k-1}$, it is easy to conclude that $\psi _{k}$ is
standard iff $\varphi _{1}=T\left[ x\right] =0$. Thus, from Proposition \ref%
{TP+-}, $T=TP_{+}$ and therefore $T\frac{d}{dx}=TP_{+}\frac{d}{dx}=T\frac{d}{%
dx}P_{-}$.

$(ii)$ From the equalities $(\tfrac{d^{2}}{dx^{2}}-q)TAu=T\tfrac{d^{2}}{%
dx^{2}}Au=T\tfrac{d}{dx}u$ and 
\begin{equation*}
TA\tfrac{d^{2}}{dx^{2}}u(x)=T\left[ u^{\prime }(x)-u^{\prime }(0)\right] =T%
\tfrac{d}{dx}u-u^{\prime }(0)T[1]
\end{equation*}%
we see that $TA$ is a transmutation operator on $\mathcal{P}(\mathbb{R})$
whenever $T[1]=0$. In this case, Proposition \ref{TP+-} tell us that $%
T=TP_{+}$, from which follows the equality $TA=TP_{+}A=TAP_{-}$.
\end{proof}

\begin{theorem}
\label{TheoTrans}Let $T$ be an s-transmutation operator on $\mathcal{P}(%
\mathbb{R})$. Then the operators $TP_{+}$, $TP_{-}$, $TAP_{+}$ and $T\frac{d%
}{dx}P_{-}$ are s-transmutations on $\mathcal{P}(\mathbb{R})$ as well.
\end{theorem}

\begin{proof}
Using the relation $\frac{d^{2}}{dx^{2}}P_{+}=P_{+}\frac{d^{2}}{dx^{2}}$
(Proposition \ref{PropP+-}) it is easy to see that $TP_{+}$ is a
transmutation operator. As the corresponding $L$-base $\psi
_{k}=TP_{+}[x^{k}]$ is such that $\psi _{k}=T[x^{k}]=\varphi _{k}$ for $k$
even, and $\psi _{k}=0$ for $k$ odd, we easily see that $\psi _{k}(0)=\psi
_{k}^{\prime }(0)=0$, $k\geq 2$. Thus $\psi _{k}$ is an s-$L$-base and $%
TP_{+}$ is an s-transmutation operator.

Let us now consider the operator $T\frac{d}{dx}P_{-}$. As $TP_{+}$ is an
s-transmutation and $TP_{+}[x]=0$, then, in accordance with Proposition \ref%
{Propothertrans} part $i)$, the operator $TP_{+}\frac{d}{dx}=T\frac{d}{dx}%
P_{-}$ is an s-transmutation as well.

By analogous means we can prove that $TP_{-}$ and $TAP_{+}$ are
s-transmutation operators.
\end{proof}

\begin{proposition}
\label{PropTphiTpsi}Let $\left\{ \varphi _{k}\right\} $ and $\left\{ \psi
_{k}\right\} $ be s-$L$-bases such that $W(\psi _{0},\psi _{1})\neq 0$, and
let $T_{\varphi }$, $T_{\psi }$ be the corresponding s-transmutation
operators satisfying $T_{\varphi }[x^{k}]=\varphi _{k}$, $T_{\psi
}[x^{k}]=\psi _{k}$. Then%
\begin{equation}
T_{\varphi }=\tfrac{W(\varphi _{0},\psi _{1})}{W(\psi _{0},\psi _{1})}%
T_{\psi }P_{+}-\tfrac{W(\varphi _{0},\psi _{0})}{W(\psi _{0},\psi _{1})}%
T_{\psi }AP_{+}+\tfrac{W(\varphi _{1},\psi _{1})}{W(\psi _{0},\psi _{1})}%
T_{\psi }\tfrac{d}{dx}P_{-}-\tfrac{W(\varphi _{1},\psi _{0})}{W(\psi
_{0},\psi _{1})}T_{\psi }P_{-}  \label{TphiTpsi}
\end{equation}%
where the equality is valid on $\mathcal{P}(\mathbb{R}).$
\end{proposition}

\begin{proof}
Let $T$ denote the operator defined on $\mathcal{P}(\mathbb{R})$ by the
right hand side of (\ref{TphiTpsi}). Note that as $\varphi _{0},\varphi
_{1},\psi _{0}$ and $\psi _{1}$ are solutions of $u^{\prime \prime }=qu$,
all the involved Wronskians on (\ref{TphiTpsi}) $W(u_{i},v_{i})$ are
constant functions and thus $T$ is an s-transmutation operator on $\mathcal{P%
}(\mathbb{R})$ since it is a linear combination of the s-transmutations $%
T_{\psi }P_{+}$, $T_{\psi }AP_{+}$, $T_{\psi }\tfrac{d}{dx}P_{-}$ and $%
T_{\psi }P_{-}$ (see Theorem \ref{TheoTrans}). As $T_{\psi }P_{-}[1]=T_{\psi
}\tfrac{d}{dx}P_{-}[1]=T_{\psi }P_{+}[x]=$ $T_{\psi }AP_{+}[x]=0$, it
follows that%
\begin{equation*}
T[1]=\tfrac{W(\varphi _{0},\psi _{1})}{W(\psi _{0},\psi _{1})}\psi _{0}-%
\tfrac{W(\varphi _{0},\psi _{0})}{W(\psi _{0},\psi _{1})}\psi _{1}=\varphi
_{0}=T_{\varphi }[1]
\end{equation*}%
and%
\begin{equation*}
T[x]=\tfrac{W(\varphi _{1},\psi _{1})}{W(\psi _{0},\psi _{1})}\psi _{0}-%
\tfrac{W(\varphi _{1},\psi _{0})}{W(\psi _{0},\psi _{1})}\psi _{1}=\varphi
_{1}=T_{\varphi }[x].
\end{equation*}%
Lastly, as $T_{\varphi }$ and $T$ are s-transmutations and $T[1]=T_{\varphi
}[1]$, $T[x]=T_{\varphi }[x]$, it follows from Corollary \ref{CoroTP+} that $%
T_{\varphi }=T$ on $\mathcal{P}(\mathbb{R})$.
\end{proof}

\section{s-Transmutation operators on Sobolev spaces}

\subsection{Existence of s-transmutation operators on W$^{2,1}(-a,a)$\label%
{SectVIOP}}

Let $q\in C[-a,a]$, $a>0$. Then, according to \cite{Marchenco}, a
transmutation operator on $C^{2}[-a,a]$ corresponding to $L=\frac{d^{2}}{%
dx^{2}}-q(x)$ can be represented in the form 
\begin{equation*}
\mathbf{T}u(x)=u(x)+\int_{-x}^{x}K(x,t)u(t)dt
\end{equation*}%
where $K(x,t)$ is constructed as follows. Consider the transformation $%
2u=x+t $, $2v=x-t$, which maps the square $\Omega :-a\leq x,y\leq a$ onto
the square $\Omega _{1}$ in the $(u,v)$-plane with vertices $%
(-a,0),(0,a),(a,0)$ and $(0,-a)$. Then, $K(x,t)=H(\frac{x+t}{2},\frac{x-t}{2}%
)=H(u,v)$ where $H(u,v)$ is the unique solution of the integral equation 
\begin{equation}
H(u,v)=\frac{1}{2}\int_{0}^{u}q(s)ds+\int_{0}^{u}\int_{0}^{v}q(\alpha +\beta
)H(\alpha ,\beta )d\beta d\alpha \text{, in }\Omega _{1}.  \label{intecH}
\end{equation}

If $q$ is $n$ times differentiable, then $K(x,t)$ is $n+1$ times
continuously differentiable \cite{Marchenco}. Moreover, if $q\in C^{1}[-a,a]$%
, then $K(x,t)\in C^{2}(\Omega )$ is the unique classical solution of the
following Goursat problem 
\begin{equation}
\left( \frac{\partial ^{2}}{\partial x^{2}}-q(x)\right) K(x,t)=\frac{%
\partial ^{2}}{\partial t^{2}}K(x,t),\quad \text{in }\Omega .  \label{GP1}
\end{equation}%
\begin{equation}
K(x,x)=\frac{1}{2}\int_{0}^{x}q(s)ds\text{, \ }K(x,-x)=0.  \label{GP2}
\end{equation}

We stress that if $q$ is differentiable, then the transmutation property of $%
\mathbf{T}$ follows from straightforward calculations, using (\ref{GP1}) and
(\ref{GP2}), \cite{KT2}. However, for a continuous $q$, since $K(x,t)\notin
C^{2}(\Omega _{1})$, the proof of this result is not trivial. In \cite%
{Marchenco}, a proof involving the use of Riemann's function was presented.
Another method was mentioned in \cite[pag.9]{LevitanInverse} without
presenting a rigorous proof: if we approximate the continuous potential $q$
by differentiable potentials $q_{n}$, then the transmutation property of $%
\mathbf{T}$ should follow by a limit extrapolation of the transmutations $%
\mathbf{T}_{n}$, say $\mathbf{T}=\lim \mathbf{T}_{n}$. This idea was also
used in \cite[proof of Th. 3.4]{KT} and \cite[proof of Th. 6]{KT1}.

In this section, we give a rigorous justification of the procedure proposed
in \cite{LevitanInverse}, and combining it with Proposition \ref{PropConvTn}%
, we prove the existence of a transmutation operator on $W^{2,1}(-a,a)$, in
the case when $q\in L^{1}(-a,a).$

\begin{proposition}
\label{PropexistH}Let $q\in L^{1}(-a,a)$. Then the integral equation (\ref%
{intecH}) has a unique solution $H(u,v)$. Moreover, $H(u,v)\in C(\Omega
_{1})\cap W^{1,1}(\Omega _{1})$ and 
\begin{equation}
\left\vert H(u,v)\right\vert \leq \left\Vert q\right\Vert
_{L^{1}}e^{a\left\Vert q\right\Vert _{L^{1}}}.  \label{estH}
\end{equation}
\end{proposition}

\begin{proof}
The proof is analogous to the one given in \cite[Th. 1.2.2]{Marchenco} for a
continuous potential $q$. For convenience, we sketch the main steps. First,
using the method of successive approximations, we write a hypothetical
series representation for the solution of (\ref{intecH}), say $%
H(u,v)=\sum\limits_{n=0}^{+\infty }H_{n}(u,v)$ where $H_{0}(u,v)=\frac{1}{2}%
\int_{0}^{u}q(s)ds$ and $n\geq 1$ 
\begin{equation}
H_{n}(u,v)=\int_{0}^{u}\int_{0}^{v}q(\alpha +\beta )H_{n-1}(\alpha ,\beta
)d\beta d\alpha .  \label{Hndef}
\end{equation}%
Next, we prove the uniform convergence of the above series, which in turn
will imply the existence and the unicity of a continuous solution of (\ref%
{intecH}). From (\ref{Hndef}), and following \cite{Marchenco}, we arrive at
the following estimate%
\begin{equation*}
\left\vert H_{n}(u,v)\right\vert \leq \left\Vert H_{0}\right\Vert _{C\left(
\Omega _{1}\right) }\frac{1}{n_{!}}\left\vert \int_{0}^{u+v}\left\vert
\int_{0}^{\beta }\left\vert q(\alpha )\right\vert d\alpha \right\vert d\beta
\right\vert ^{n}\leq \left\Vert q\right\Vert _{L^{1}}\frac{\left(
a\left\Vert q\right\Vert _{L^{1}}\right) ^{n}}{n!}
\end{equation*}%
From this and by the Weierstrass M-test it follows that the series defining $%
H$ is uniformly convergent and $\left\vert H(u,v)\right\vert \leq
\tsum\limits_{n=0}^{+\infty }\left\vert H_{n}(u,v)\right\vert \leq
\left\Vert q\right\Vert _{L^{1}}e^{a\left\Vert q\right\Vert _{L^{1}}}$.
Lastly, from (\ref{intecH}) we see that $H\in W^{1,1}(\Omega _{1})$ as it is
the sum of the absolutely continuous function $\frac{1}{2}\int_{0}^{u}q(s)ds$
with a function in $C^{1}(\Omega _{1})$.
\end{proof}

\begin{proposition}
\label{q to H}Let $q\in L^{1}(-a,a)$ and let $H(u,v)$ be the unique solution
of (\ref{intecH}). Then the mapping $q\rightarrow H(u,v)$ is continuous from 
$L^{1}(-a,a)$ to $C(\Omega _{1})$.
\end{proposition}

\begin{proof}
Denote by $\widetilde{H}(u,v)$ the solution of (\ref{intecH}) with $q$
replaced by $\widetilde{q}\in L^{1}(-a,a)$. It suffices to prove that the
inequality 
\begin{equation*}
\left\vert \widetilde{H}(u,v)-H(u,v)\right\vert \leq C\left\Vert \widetilde{q%
}-q\right\Vert _{L^{1}}e^{a\left\Vert \widetilde{q}\right\Vert _{L^{1}}}
\end{equation*}%
holds for all $q,\widetilde{q}\in L^{1}(-a,a)$ with some constant $C>0$ not
depending on $\widetilde{q}$. In order to prove this, we first observe that $%
M(u,v):=\widetilde{H}(u,v)-H(u,v)$ is the unique solution of the integral
equation $M(u,v)=M_{0}(u,v)+\int_{0}^{u}\int_{0}^{v}\widetilde{q}(\alpha
+\beta )M(\alpha ,\beta )d\beta d\alpha $ where 
\begin{equation}
M_{0}(u,v)=\frac{1}{2}\int_{0}^{u}\left( \widetilde{q}(s)-q(s)\right)
ds+\int_{0}^{u}\int_{0}^{v}H(\alpha ,\beta )\left( \widetilde{q}(\alpha
+\beta )-q(\alpha +\beta )\right) d\beta d\alpha \text{.}  \label{Mcero}
\end{equation}%
Now, applying the same reasoning as in the proof of the above proposition,
one can check that $\left\vert M(u,v)\right\vert \leq \left\Vert
M_{0}\right\Vert _{C\left( \Omega _{1}\right) }e^{a\left\Vert \widetilde{q}%
\right\Vert _{L^{1}}}$. On the other hand, from (\ref{Mcero}) we get%
\begin{equation*}
\left\vert M_{0}(u,v)\right\vert \leq \frac{1}{2}\left\Vert \widetilde{q}%
-q\right\Vert _{L^{1}}+a\underset{\Omega _{1}}{\sup }\left\vert H(\alpha
,\beta )\right\vert \left\Vert \widetilde{q}-q\right\Vert
_{L^{1}}=C\left\Vert \widetilde{q}-q\right\Vert _{L^{1}}.
\end{equation*}%
Now, combining the last two inequalities, the desired result follows.
\end{proof}

\begin{theorem}
\label{Prop_exist_transm}Let $q\in L^{1}(-a,a)$ and let $H(u,v)$ be the
unique solution of (\ref{intecH}). Then, $K(x,t)=H(\frac{x+t}{2},\frac{x-t}{2%
})$ is a weak solution of the Goursat problem (\ref{GP1}), (\ref{GP2}) and%
\begin{equation}
\mathbf{T}u(x)=u(x)+\int_{-x}^{x}K(x,t)u(t)dt  \label{VIOp}
\end{equation}%
is an s-transmutation operator on $W^{2,1}(-a,a)$, corresponding to $L=\frac{%
d^{2}}{dx^{2}}-q(x)$. The corresponding $L$-base $\alpha _{k}(x):=\mathbf{T}%
\left[ x^{k}\right] $ is such that%
\begin{equation}
\alpha _{0}(0)=\alpha _{1}^{\prime }(0)=1\text{ and }\alpha _{0}^{\prime
}(0)=\alpha _{1}(0)=0.  \label{icalphak}
\end{equation}%
Moreover, $\mathbf{T}$ is invertible and $\mathbf{T}^{-1}u(x)=u(x)-%
\dint_{-x}^{x}K(t,x)u(t)dt$.
\end{theorem}

\begin{proof}
Let $q_{n}\in C^{1}[-a,a]$ be a sequence of functions that converges to $q$
in $L^{1}(-a,a)$ and let $H_{n}(u,v)$ be the unique solution of the integral
equation%
\begin{equation}
H_{n}(u,v)=\frac{1}{2}\int_{0}^{u}q_{n}(s)ds+\int_{0}^{u}\int_{0}^{v}q_{n}(%
\alpha +\beta )H_{n}(\alpha ,\beta )d\beta d\alpha \text{.}  \label{intecHn}
\end{equation}%
Since $q_{n}\in C^{1}[-a,a]$, we have that $K_{n}(x,t):=H_{n}(\frac{x+t}{2},%
\frac{x-t}{2})$ is a classical a solution of the equation $(\partial
_{x}^{2}-\partial _{t}^{2}-q_{n}(x))K_{n}(x,t)=0$ (see the beginning of this
section). In particular, the equality%
\begin{equation}
\int \int_{\Omega }\left[ \left( \partial _{t}K_{n}\right) \partial
_{t}\varphi -\left( \partial _{x}K_{n}\right) \partial _{x}\varphi
-q_{n}K_{n}\varphi \right] dxdt=0,  \label{weakKn}
\end{equation}%
holds for all $\varphi \in C_{c}^{\infty }(\Omega )$. As $q_{n}\rightarrow q$
in $L^{1}(-a,a)$, it follows from Proposition \ref{q to H} that $%
H_{n}\rightarrow H$ uniformly on $\Omega _{1}$. Moreover, differentiating (%
\ref{intecHn}) and using the uniform convergence of $H_{n}$, we see that $%
\partial _{u}H_{n}\rightarrow \partial _{u}H$ and $\partial
_{v}H_{n}\rightarrow \partial _{v}H$ in $L^{1}(\Omega _{1})$ as $%
n\rightarrow \infty $. Thus $K_{n}\rightarrow K$ uniformly on $\Omega $ and $%
\partial _{x}K_{n}\rightarrow \partial _{x}K$, $\partial
_{y}K_{n}\rightarrow \partial _{y}K$ in $L^{1}(\Omega )$ as $n\rightarrow
\infty $. Then, letting $n\rightarrow +\infty $ in (\ref{weakKn}), we obtain%
\begin{equation*}
\int \int_{\Omega }\left[ \left( \partial _{t}K\right) \partial _{t}\varphi
-\left( \partial _{x}K\right) \partial _{x}\varphi -qK\varphi \right] dxdt=0
\end{equation*}%
for all $\varphi \in C_{c}^{\infty }(\Omega )$. Thus $K$ is a weak solution
of (\ref{GP1})$.$ The boundary conditions (\ref{GP2}) for $K$ are obtained
by taking the limit of the corresponding boundary conditions for $K_{n}$, $%
K_{n}(x,x)=\frac{1}{2}\int_{0}^{x}q_{n}(s)ds,$ $K_{n}(x-x)=0$. Thus $K$ is a
weak solution of (\ref{GP1}), (\ref{GP2}).

The operator $\mathbf{T}$ is an s-transmutation because it is the uniform
limit of the s-transmutation operators $\mathbf{T}_{n}u(x)=u(x)+%
\int_{-x}^{x}K_{n}(x,t)u(t)dt$, see Proposition \ref{PropConvTn}. The
initial conditions (\ref{icalphak}) of the $L$-base $\left\{ \alpha
_{k}\right\} $ follow by differentiating (\ref{VIOp}) and then using (\ref%
{GP2}) at $x=0$.

The formula for the kernel of $T^{-1}$ was obtained in \cite{KT1} for a
continuous potential $q$ and, by using a limit procedure as above, it is
easy to see that it remains true for an integrable $q$. The theorem is
proved.
\end{proof}

\subsection{A fundamental set of s-transmutation operators}

Let $\mathbf{T}$ be the s-transmutation operator (\ref{VIOp}) studied in
Proposition \ref{Prop_exist_transm}. In accordance with Theorem \ref%
{TheoTrans}\, each one of the operators%
\begin{equation}
\mathbf{T}P_{+},\text{ }\mathbf{T}P_{-},\text{ }\mathbf{T}AP_{+},\text{ }%
\mathbf{T}\tfrac{d}{dx}P_{-},\text{ }  \label{fund}
\end{equation}%
is an s-transmutation on $\mathcal{P}(\mathbb{R})$ (the linear space of
polynomials). Next we show that (\ref{fund}) is a fundamental set of
s-transmutation operators in the sense that any s-transmutation operator is
a linear combination of operators from (\ref{fund}). Set $I=(-a,a)$, $%
0<a<\infty .$

\begin{theorem}
\label{Fulltransmu}Let $\left\{ \varphi _{k}\right\} $, $k\in \mathbb{N}_{0}$%
, be an s-$L$-base. Then there exists $T_{\varphi }$ an s-transmutation
operator on $W^{3,1}(I)$ such that $T_{\varphi }[x^{k}]=\varphi _{k}$. The
equality%
\begin{equation}
T_{\varphi }=\varphi _{0}(0)\mathbf{T}P_{+}+\varphi _{0}^{\prime }(0)\mathbf{%
T}AP_{+}+\varphi _{1}(0)\mathbf{T}\tfrac{d}{dx}P_{-}+\varphi _{1}^{^{\prime
}}(0)\mathbf{T}P_{-}  \label{reltransm}
\end{equation}%
holds on $W^{1,1}(I)$, where the operators $\mathbf{T}$, $P_{\pm }$, and $A$
are given by (\ref{VIOp}) and (\ref{Proye}). Also, $T_{\varphi
}:W^{1,1}(I)\rightarrow L^{1}(I)$ is bounded. Moreover, if $W(\varphi
_{0},\varphi _{1})\neq 0$, then the operator%
\begin{equation*}
M_{\varphi }:=\frac{1}{W(\varphi _{0},\varphi _{1})}\left\{ \varphi
_{1}^{^{\prime }}(0)P_{+}\mathbf{T}^{-1}-\varphi _{0}^{^{\prime }}(0)AP_{+}%
\mathbf{T}^{-1}-\varphi _{1}(0)\tfrac{d}{dx}P_{-}\mathbf{T}^{-1}+\varphi
_{0}(0)P_{-}\mathbf{T}^{-1}\right\}
\end{equation*}%
satisfies $T_{\varphi }M_{\varphi }(u)=M_{\varphi }T_{\varphi }(u)=u$ for
all $u\in W^{2,1}(I)$. In particular, if $\varphi _{1}(0)=0$, then $%
T_{\varphi }:L^{1}(I)\rightarrow L^{1}(I)$ is bounded, and if in addition $%
\varphi _{0}(0)\neq 0$, then $T_{\varphi }$ is invertible with $T_{\varphi
}^{-1}=M_{\varphi }$.
\end{theorem}

\begin{proof}
As $\mathbf{T}$ and $T_{\varphi }$ are both s-transmutations on $\mathcal{P}(%
\mathbb{R})$ with $T_{\varphi }[x^{k}]=\varphi _{k}$ and $\mathbf{T}%
[x^{k}]=\alpha _{k}$, from Proposition \ref{PropTphiTpsi} and using (\ref%
{icalphak}) we see that (\ref{reltransm}) holds on $\mathcal{P}(\mathbb{R})$%
. As $\mathbf{T}:L^{1}(I)\rightarrow L^{1}(I)$ is bounded, the operator
defined on $\mathcal{P}(\mathbb{R})$ by the right hand side of (\ref%
{reltransm}) can be extended continuously to $W^{1,1}(I)$, from which we
conclude that $T_{\varphi }:W^{1,1}(I)\rightarrow L^{1}(I)$ is bounded. Thus
employing Theorem \ref{Thmapping} and Remark \ref{ReThmapping}, $T_{\varphi
} $ is an s-transmutation operator on $W^{3,1}(I).$ The relation $T_{\varphi
}M_{\varphi }(u)=M_{\varphi }T_{\varphi }(u)=u$ is obtained after a few but
trivial calculations, and using the existence of $\mathbf{T}^{-1}$ (see
Proposition \ref{Prop_exist_transm}). All the remaining affirmations are
obvious.
\end{proof}

\begin{corollary}
\label{CorSPPS}Let $\left\{ \varphi _{k}\right\} $ be an s-$L$-base such
that $W(\varphi _{0},\varphi _{1})\neq 0$ and $\lambda \in \mathbb{C}$. Then
the general weak solution of the equation%
\begin{equation}
v^{\prime \prime }+qv=\lambda v  \label{SL_eq}
\end{equation}%
on $I$ has the form $v_{ger}=c_{1}v_{1}+c_{2}v_{2}$ where $c_{1}$ and $c_{2}$
are arbitrary complex constants,%
\begin{equation}
v_{1}=\sum\limits_{k=0}^{\infty }\lambda ^{k}\frac{\varphi _{2k}}{(2k)!}%
\text{ \ and \ }v_{2}=\sum\limits_{k=1}^{\infty }\lambda ^{k}\frac{\varphi
_{2k+1}}{(2k+1)!}\text{,}  \label{SPPS}
\end{equation}%
and both series, as well as the corresponding series of derivatives,
converge uniformly on $\overline{I}$.
\end{corollary}

\begin{proof}
Let $T_{\varphi }$ be the s-transmutation operator studied in Theorem \ref%
{Fulltransmu} and satisfying $T_{\varphi }[x^{k}]=\varphi _{k}$. From the
transmutation property of $T_{\varphi }$, we can see that $(\frac{d^{2}}{%
dx^{2}}-q-\lambda )T_{\varphi }u=T_{\varphi }(\frac{d^{2}}{dx^{2}}-\lambda
)u $ for all $\lambda \in \mathbb{C}$ and $u\in W^{3,1}(I)$. Thus, the
operator $T_{\varphi }$ maps classical solutions of $u^{\prime \prime
}=\lambda u$ to weak solutions of (\ref{SL_eq}). Moreover, as $W(\varphi
_{0},\varphi _{1})\neq 0$, it follows that $T_{\varphi }$ is one-to-one on $%
W^{3,1}(I)$, and thus the general solution of (\ref{SL_eq}) is obtained by
applying $T_{\varphi }$ to the general solution of $u^{\prime \prime
}=\lambda u$, $u_{ger}=c_{1}\cosh (\lambda x)+c_{2}(1/\sqrt{\lambda })\sinh (%
\sqrt{\lambda }x)$. Thus, $v_{ger}=T_{\varphi }[u_{ger}]=c_{1}T_{\varphi
}[\cosh (\sqrt{\lambda }x)]+c_{2}T_{\varphi }[(1/\sqrt{\lambda })\sinh (%
\sqrt{\lambda }x)]$ where%
\begin{equation*}
T_{\varphi }[\cosh (\lambda x)]=T_{\varphi }\left[ \sum_{k=0}^{+\infty
}\lambda ^{k}\frac{x^{2k}}{(2k)!}\right] =\sum_{k=0}^{+\infty }\lambda ^{k}%
\frac{\varphi _{2k}}{(2k)!}=v_{1}
\end{equation*}%
and $T_{\varphi }[(1/\sqrt{\lambda })\sinh (\sqrt{\lambda }x)]=T_{\varphi }%
\left[ \sum_{k=0}^{+\infty }\lambda ^{k}\frac{x^{2k+1}}{(2k+1)!}\right]
=\sum_{k=0}^{+\infty }\lambda ^{k}\frac{\varphi _{2k+1}}{(2k+1)!}=v_{2}$.
Lastly, the convergence of the series defining the solutions $v_{1}$ and $%
v_{2}$ follows from the fact that $T_{\varphi }$ is bounded from $C^{2}(%
\overline{I})$ into $C^{1}(\overline{I})$.
\end{proof}

\begin{remark}
\label{Trans_Relax}Let us relate the results we have obtained so far with
the ones appearing in the literature. Assume that $q\in C\left( \overline{I}%
\right) $ and that equation (\ref{L-base0,1}) possesses a nonvanishing
solution $f$ $\in C^{2}(\overline{I})$. Set $\varphi _{0}:=f$ and note that
a second linearly independent solution can be computed by the formula $%
\varphi _{1}(x)=f(x)\int_{0}^{x}\frac{dt}{f^{2}(t)}$. Without loss of
generality assume that $f(0)=1$ and set $c:=f^{\prime }(0)$. Thus, $\varphi
_{0}(0)=1$, $\varphi _{0}^{\prime }(0)=c$, $\varphi _{1}(0)=0$, and $\varphi
_{1}^{\prime }(0)=1$. Let $\left\{ \varphi _{k}\right\} $ be the s-$L$-base
corresponding to this pair of solutions, and denote by $T_{c}$ the
(transmutation) operator such that $T_{c}\left[ x^{k}\right] =\varphi _{k}$.
Notice that $T_{0}\equiv \mathbf{T}$. The operator $T_{c}$ has been studied
in several places in the literature. In \cite{CKT}, assuming $q\in C^{1}(%
\overline{I})$, it was proved that $T_{c}$ is a Volterra integral operator
and also a transmutation operator on $C^{2}(\overline{I})$. Another proof of
the transmutation property of $T_{c}$, for any value of $c\in \mathbb{C}$,
was given in \cite{KT1} when assuming that $q\in C^{1}(\overline{I})$ and
also on \cite{KT2} \ supposing that $q\in C(\overline{I})$ and that $\varphi
_{0}(x)$ is nonvanishing on $\overline{I}$. In \cite{CKT} a formula for the
kernel of the operator $T_{c}$ was obtained in terms of the kernel of the
same operator with $c=0$. Starting from this result, the relation $T_{c}u=%
\mathbf{T}\left[ u(x)+\frac{c}{2}\tint_{-x}^{x}u(t)dt\right] $ was obtained
in \cite{KT2}. This equality is the one given by (\ref{reltransm}), in this
special case. In \cite{KT}, and with the help of the above relation, it was
proved that $T_{c}$ is a transmutation operator for any $q\in C(\overline{I}%
) $ and for any value of $c$. Perhaps at the same time, another proof for
this result was presented in \cite{CamposPhd}, which is published here (the
proof of Theorem \ref{Thmapping}) for the first time.

The spectral parameter power series representation, like those representing
the general solution of (\ref{SL_eq}) in Corollary \ref{CorSPPS}, was first
introduced in \cite{KravSPPS} as an application of pseudoanalytic function
theory and nowadays represents a useful tool for the numerical solution of
both regular and singular Sturm--Liouville problems, see \cite{BCK}, \cite%
{KT} and the references therein. Originally, the coefficients of the power
series (\ref{SPPS}) were computed by slightly different recursive formulas,
and its construction required the existence of a nonvanishing solution of (%
\ref{SL_eq}) with $\lambda =0$. Here we have relaxed this restriction using
the s-$L$-bases framework.
\end{remark}

\section{s-Transmutation operators on distribution spaces}

Let $I=(-a,a)$, $0<a<\infty $, and let $q\in L_{loc}^{1}(I)$. A linear
operator $T:D_{T}\subset \mathcal{D}^{\prime }(I)\rightarrow \mathcal{D}%
^{\prime }(I)$ is called a transmutation operator on $\chi \subset \mathcal{D%
}^{\prime }(I)$ if for every $u\in \chi $

\begin{equation*}
\left( \frac{d^{2}}{dx^{2}}-q\right) Tu=T\frac{d^{2}}{dx^{2}}u\quad \text{in 
}\mathcal{D}^{\prime }(I).
\end{equation*}%
Our next goal is to show that the transmutation operators constructed in
Section 5 can be continuously extended from Sobolev spaces to suitable
spaces of distributions, preserving the transmutation property. To this, let
us consider the s-transmutation operator studied in Section \ref{SectVIOP},
Theorem \ref{Prop_exist_transm}, 
\begin{equation}
\mathbf{T}u(x)=u(x)+\int_{-x}^{x}K(x,t)u(t)dt.  \label{VIoploc}
\end{equation}%
Since $q\in L_{loc}^{1}(I)$, we have $K(x,t)\in C(I\times I)\cap
W_{loc}^{1,1}(I\times I)$ and thus $\mathbf{T}$ maps $L_{loc}^{1}(I)$ into
itself. In order to extend $\mathbf{T}$ to distributions, using Proposition %
\ref{PropLext}, we first have to look for the existence of the corresponding
transpose operator. An easy computation shows that%
\begin{equation}
\mathbf{T}^{\square }\psi (x)=\psi (x)-\int_{-a}^{-\left\vert x\right\vert
}K(t,x)\psi (t)dt+\int_{\left\vert x\right\vert }^{a}K(t,x)\psi (t)dt,\text{%
\quad }\psi \in C_{c}(I),  \label{Ttransp}
\end{equation}%
with the equality $\int_{I}(\mathbf{T}u)\phi dx=\int_{I}u(\mathbf{T}%
^{\square }\phi )dx$ valid for all $u\in L_{loc}^{1}(I)$, $\phi \in
C_{c}(I). $

Let us treat first the case when $q\in C^{\infty }(I)$.

\begin{proposition}
\label{PropTransExt}Suppose $q\in C^{\infty }(I)$. If the operator $T$ is
continuous on $\mathcal{D}^{\prime }(I)$, and a transmutation on $%
C_{c}^{\infty }(I)$, then $T$ is a transmutation operator on $\mathcal{D}%
^{\prime }(I)$.
\end{proposition}

\begin{proof}
Let $u\in \mathcal{D}^{\prime }(I)$. As $C_{c}^{\infty }(I)$ is dense in $%
\mathcal{D}^{\prime }(I)$ \cite[Th. 3.18]{Grubb}, there exists a sequence $%
\phi _{n}\in C_{c}^{\infty }(I)$ such that $\phi _{n}\rightarrow u$ in $%
\mathcal{D}^{\prime }(I)$ as $n\rightarrow \infty $. Hence,%
\begin{equation*}
\left( \dfrac{d^{2}}{dx^{2}}-q\right) \mathbf{T}u=\lim_{n\rightarrow \infty
}\left( \dfrac{d^{2}}{dx^{2}}-q\right) \mathbf{T}p_{n}=\lim_{n\rightarrow
\infty }\mathbf{T}\dfrac{d^{2}}{dx^{2}}p_{n}=\mathbf{T}\dfrac{d^{2}}{dx^{2}}%
u,\text{ \ in }\mathcal{D}^{\prime }(I),
\end{equation*}%
where we have used the continuity in $\mathcal{D}^{\prime }(I)$ of the
involved operators and the transmutation property of $T$ on $C_{c}^{\infty
}(I)$.
\end{proof}

\bigskip

\begin{proposition}
\label{PropTdist}Assume that $q\in C^{\infty }(I)$ and let $\mathbf{T}$ be
the linear operator acting on distributions according to the rule%
\begin{equation}
\left\langle \mathbf{T}u,\phi \right\rangle :=\left\langle u,\mathbf{T}%
^{\square }\phi \right\rangle \text{,\quad }u\in \mathcal{D}^{\prime
}(I),\phi \in C_{c}^{\infty }(I)  \label{Tdist}
\end{equation}%
where $\mathbf{T}^{\square }$ is given by (\ref{Ttransp}). Then, $\mathbf{T}:%
\mathcal{D}^{\prime }(I)\rightarrow \mathcal{D}^{\prime }(I)$ is continuous
and invertible. Also, $\mathbf{T}$ is a transmutation operator on $\mathcal{D%
}^{\prime }(I)$.
\end{proposition}

\begin{proof}
First let us prove that $\mathbf{T}$ is well defined and continuous on $%
\mathcal{D}^{\prime }(I)$. According to Proposition \ref{PropLext}, it
suffices to show that $\mathbf{T}^{\square }$ is continuous on $%
C_{c}^{\infty }(I)$. Suppose the sequence $\phi _{n}$ converge in $%
C_{c}^{\infty }(I)$ to $\phi $. Then there exists a number $b<a$ such that $%
\func{supp}\phi _{n}\subseteq \lbrack -b,b]\subset I$ for all $n,$ and from (%
\ref{Ttransp}) we see that $\func{supp}(\mathbf{T}^{\square }\phi
_{n})\subseteq \lbrack -b,b]\subset I$ for all $n$. On the other hand, as $%
q\in C^{\infty }(I)$, we have $K(x,t)\in C^{\infty }(I\times I)$ and thus
the expression (\ref{Ttransp}) defining $\mathbf{T}^{\square }\phi $ can be
differentiated infinitely many times whenever $\phi \in C_{c}^{\infty }(I)$.
For example,%
\begin{equation}
\begin{array}{c}
\dfrac{d}{dx}\mathbf{T}^{\square }\phi (x)=\dfrac{d}{dx}\phi
(x)-\dint_{-a}^{-\left\vert x\right\vert }\partial _{x}K(t,x)\phi
(t)dt+\dint_{\left\vert x\right\vert }^{a}\partial _{x}K(t,x)\phi
(t)dt\bigskip  \\ 
+K(-x,x)\phi (-x)+K(x,x)\phi (x).%
\end{array}
\label{derivTtransp}
\end{equation}%
From the above facts it is easy to see that $\mathbf{T}^{\square }$ maps $%
C_{c}^{\infty }(I)$ into itself and that $\mathbf{T}^{\square }\phi
_{n}\rightarrow \mathbf{T}^{\square }\phi $ in $C_{c}^{\infty }(I)$ whenever 
$\phi _{n}\rightarrow \phi $ in $C_{c}^{\infty }(I),$ as $n\rightarrow
\infty $. This proves the continuity of $\mathbf{T}^{\square }$.

Since $\mathbf{T}$ is invertible on $L_{loc}^{1}(I)$ and $\mathbf{T}^{-1}$
is a Volterra integral operator (Theorem \ref{Prop_exist_transm}), it
follows that $\mathbf{T}^{\square }$ is invertible as well, and $\left( 
\mathbf{T}^{\square }\right) ^{-1}=\left( \mathbf{T}^{-1}\right) ^{\square }$%
. Therefore the inverse operator of $\mathbf{T}$ on $\mathcal{D}^{\prime
}(I) $ is defined by the rule $\left\langle \mathbf{T}^{-1}u,\phi
\right\rangle :=\left\langle u,\left( \mathbf{T}^{\square }\right) ^{-1}\phi
\right\rangle $ for $u\in \mathcal{D}^{\prime }(I)$ and $\phi \in C^{\infty
}(I)$.

Lastly, the transmutation property of $\mathbf{T}$ on $\mathcal{D}^{\prime
}(I)$ follows from Proposition \ref{PropTransExt} upon observing that $%
\mathbf{T}$ is continuous on $\mathcal{D}^{\prime }(I)$ and a transmutation
on $C_{c}^{\infty }(I)$ (Theorem \ref{Prop_exist_transm}).
\end{proof}

\begin{corollary}
\label{Cor_dist=clas}Let $q\in C^{\infty }(I)$. Then the differential
equation $u^{\prime \prime }=qu$, $u\in \mathcal{D}^{\prime }(I)$ has only
classical solutions.
\end{corollary}

\begin{proof}
Let $u\in \mathcal{D}^{\prime }(I)$ be a solution of $u^{\prime \prime }=qu$%
. Then, from the transmutation property of $\mathbf{T}$ it follows that $v=%
\mathbf{T}^{-1}u$ is a solution of $v^{\prime \prime }=0$, $v\in \mathcal{D}%
^{\prime }(I)$. The last equation has only classical solutions \cite{Neto},
thus $v=C_{1}+C_{2}x$ for some constants $C_{1}$, $C_{2}$ and therefore $u=%
\mathbf{T}v=C_{1}\mathbf{T}[1]+C_{2}\mathbf{T}[x]\in C^{\infty }(I)$ is a
classical solution of the first equation.
\end{proof}

\begin{corollary}
\label{CorTureg}Let $q\in C^{\infty }(I)$. If $\mathcal{T}$ is a
transmutation operator on $\mathcal{D}^{\prime }(I)$, then $\mathcal{T}\left[
x^{k}\right] $, $k\in \mathbb{N}_{0}$, is a regular distribution from the
space $C^{\infty }(I)$.
\end{corollary}

\begin{proof}
Set $\varphi _{k}:=\mathcal{T}\left[ x^{k}\right] \in \mathcal{D}^{\prime
}(I)$ for $k\in \mathbb{N}_{0}$. From the transmutation property of $%
\mathcal{T} $ it follows that $\varphi _{k}$ is a distributional solution of
(\ref{L-base0,1}) when $k=0,1$. For $k\geq 2$, it satisfies (\ref%
{L-baseplus2}). However, from Corollary \ref{Cor_dist=clas} we see that
Equations (\ref{L-base0,1}) and (\ref{L-baseplus2}) have only classical
solutions, thus $\mathcal{T}\left[ x^{k}\right] \in C^{\infty }(I)$.
\end{proof}

\bigskip

The last result allows one to consider the notion of standard transmutation
operator (Definition \ref{standard}) on spaces of distributions.

In what follows, we obtain the analogue of some results of Section \ref%
{Sect4} for distributions. Consider the operators $P_{\pm }u(x)=\frac{%
u(x)\pm u(x)}{2}$ and $Au(x)=\tint_{0}^{x}u(t)dt$ acting on $L_{loc}^{1}(I)$%
. Since $P_{\pm }^{\square }=P_{\pm }$ is continuous on $C_{c}^{\infty }(I)$%
, it follows from Proposition \ref{PropLext} that $P_{\pm }$ is well defined
and continuous on $\mathcal{D}^{\prime }(I)$. A similar extension is not
available for the operator $A$. Indeed, 
\begin{equation*}
A^{\square }\psi (x)=\left\{ 
\begin{array}{cc}
-\int_{-a}^{x}\psi (t)dt & ,-a\leq x\leq 0 \\ 
\int_{x}^{a}\psi (t)dt & 0<x\leq a%
\end{array}%
\right. .
\end{equation*}%
Thus if $\psi \in C_{c}^{\infty }(I)$, we can not be assured of the
continuity of the function $A^{\square }\psi (x)$ at $x=0$. However, it is
easy to see that $\left( AP_{+}\right) ^{\square }=(P_{-}A)^{\square
}=A^{\square }P_{-}^{\square }=A^{\square }P_{-}$ is continuous on $%
C_{c}^{\infty }(I)$, and thus the operator $AP_{+}$ admits a continuous
extension to $\mathcal{D}^{\prime }(I)$.

\begin{proposition}
\label{FundSDist}Let $q\in C^{\infty }(I)$ and let $\mathbf{T}$ be the
operator defined in $\mathcal{D}^{\prime }(I)$ by (\ref{VIoploc})-(\ref%
{Tdist}). Then, a continuous linear operator $\mathcal{T}:\mathcal{D}%
^{\prime }(I)\rightarrow \mathcal{D}^{\prime }(I)$ is an s-transmutation on $%
\mathcal{D}^{\prime }(I)$ if and only if $\mathcal{T}$ is a linear
combination of operators from (\ref{fund}).
\end{proposition}

\begin{proof}
Once we have shown that each operator in (\ref{fund}) is a continuous
s-transmutation on $\mathcal{D}^{\prime }(I)$, then the same will be true
for all their linear combinations. The continuity follows from the fact that
each one of the operators $\mathbf{T}$, $P_{\pm }$, $\frac{d}{dx}$, and $%
AP_{+}$ is continuous on $\mathcal{D}^{\prime }(I)$, and the transmutation
property on $\mathcal{D}^{\prime }(I)$ is a consequence of Theorem \ref%
{Fulltransmu} and Proposition \ref{PropTransExt}.

Conversely, let $\mathcal{T}$ be a continuous s-transmutation operator on $%
\mathcal{D}^{\prime }(I)$. From Corollary \ref{CorTureg} we see that the
corresponding $L$-base, $\varphi _{k}:=\mathcal{T}\left[ x^{k}\right] $ is
such that $\varphi _{k}\in C^{\infty }(I)$. Let us set 
\begin{equation*}
T_{\varphi }:=\varphi _{0}(0)\mathbf{T}P_{+}+\varphi _{0}^{\prime }(0)%
\mathbf{T}AP_{+}+\varphi _{1}(0)\mathbf{T}\tfrac{d}{dx}P_{-}+\varphi
_{1}^{\prime }(0)\mathbf{T}P_{-}.
\end{equation*}%
Note that by Theorem \ref{Fulltransmu}, we have $T_{\varphi }\left[ x^{k}%
\right] =\varphi _{k}$. Thus, $T_{\varphi }\left[ x^{k}\right] =\mathcal{T}%
\left[ x^{k}\right] $ for every $k\in \mathbb{N}_{0}$, and as both operators
are linear they are equal on polynomials. In fact, as the set of polynomials
is dense in $\mathcal{D}^{\prime }(I)$ and both operators are continuous on $%
\mathcal{D}^{\prime }(I)$, we see that $T_{\varphi }$ and $\mathcal{T}$ are
equal on $\mathcal{D}^{\prime }(I)$. This shows that $\mathcal{T\equiv }%
T_{\varphi }$ is a linear combination of operators from (\ref{fund})$.$
\end{proof}

\bigskip

Now assume that $q\notin C^{\infty }(I)$. In such a case, the domain of the
operator $\mathbf{T}$ defined in (\ref{VIoploc}) is smaller than $\mathcal{D}%
^{\prime }(I)$ and it depends on the smoothness of the kernel $K(x,t)$,
which in turn depends on the smoothness of the potential $q$. Denote by $%
F(I) $, $I=(-a,a)$, any one of the following spaces $L^{p}(I)$, $C^{k}(I)$, $%
k\in \mathbb{N}_{0}$, $1\leq p\leq \infty $. We are going to consider $F(I)$
with the convergence induced by the inductive limit topology (see Section %
\ref{SectionPre}). Define $W^{2}(I;F)=\left\{ \phi \in W^{1,1}(I):\phi
^{\prime \prime }\in F(I)\right\} $. If $K\subset I$ is compact, then $%
W^{2}(K;F)$ is a Banach space with norm $\left\Vert u\right\Vert
_{W^{2}(K,F)}=\left\Vert u\right\Vert _{W^{1,1}(K)}+\left\Vert u\right\Vert
_{F(k)}$. We shall also consider $W_{c}^{2}(I;F)$ with the inductive limit
topology.

The extension of an operator $T$ from functions to functionals will be based
once more on the identity $\left\langle Tu,\phi \right\rangle =\left\langle
u,T^{\square }\phi \right\rangle $, whenever it makes sense.

\begin{proposition}
Assume that $q\in F_{loc}(I)$. Then the operator $\mathbf{T}$ defined in (%
\ref{VIoploc}) is continuous and invertible on $\left( W_{c}^{2}(I;F)\right)
^{\prime }$ as well on $\left( F_{c}(I)\right) ^{\prime }$. Moreover, $%
\mathbf{T}$ is a transmutation operator on $\left( F_{c}(I)\right) ^{\prime
} $.
\end{proposition}

\begin{proof}
The continuity of $\mathbf{T}$ on $\left( W_{c}^{2}(I;F)\right) ^{\prime }$
follows from the continuity of $\mathbf{T}^{\square }$ on $W_{c}^{2}(I;F)$,
and its continuity on $\left( F_{c}(I)\right) ^{\prime }$ follows from the
continuity of $\mathbf{T}^{\square }$ on $F_{c}(I)$. Let us prove this when $%
q\in L_{loc}^{p}(I)$, thus $F(I)=L^{p}(I)$ and $%
W_{c}^{2}(I;F)=W_{c}^{2,p}(I) $. The other cases will follow by analogy.
From (\ref{Ttransp}) and since $K(t,x)\in C(I\times I)\cap
W_{loc}^{1,1}(I\times I)$, it is clear that $\mathbf{T}^{\square }$ is
continuous on $L_{c}^{p}(I)$. Moreover, if $\phi \in W_{c}^{2,p}(I)$, then
formula (\ref{derivTtransp}) holds, and from this we see that $\mathbf{T}%
\phi \in W_{c}^{1,p}(I)$. On the other hand, if $\phi ,\psi \in
W_{c}^{2,p}(I)$, then integration by parts gives 
\begin{eqnarray*}
\dint\limits_{I}(T^{\square }\phi )^{\prime }\psi ^{\prime }dx
&=&-\dint\limits_{I}(T^{\square }\phi )\psi ^{\prime \prime
}dx=-\dint\limits_{I}\phi T\psi ^{\prime \prime }dx= \\
&=&-\dint\limits_{I}\phi \left\{ (T\psi )^{\prime \prime }-qT\psi \right\}
dx=-\dint\limits_{I}\left\{ T^{\square }(\phi ^{\prime \prime }-q\phi
)\right\} \psi dx,
\end{eqnarray*}%
where we have used the fact that the functions $\mathbf{T}^{\square }\phi $
and $\mathbf{T}^{\square }\psi $ have compact support and also the
transmutation property of $\mathbf{T}$ on $W_{c}^{2,1}(I)$. Thus we have
shown that $\mathbf{T}^{\square }\phi \in W_{c}^{2,p}(I)$ and 
\begin{equation}
(\mathbf{T}^{\square }\phi )^{\prime \prime }=\mathbf{T}^{\square }(\phi
^{\prime \prime }-q\phi ),\ \phi \in W_{c}^{2,p}(I).  \label{secderadjoint}
\end{equation}%
Now, by (\ref{derivTtransp}) and (\ref{secderadjoint}) we see that $\mathbf{T%
}^{\square }\varphi _{n}\rightarrow \mathbf{T}^{\square }\varphi $ in $%
W_{c}^{2,p}(I)$, whenever $\varphi _{n}\rightarrow \varphi $ in $%
W_{c}^{2,p}(I)$, as $n\rightarrow \infty $. This ensures the continuity of $%
\mathbf{T}^{\square }$ on $W_{c}^{2,p}(I)$.

By analogy with the proof of Theorem \ref{PropTdist}, we arrive at the
invertibility of $\mathbf{T}$.

Lastly, let us prove the transmutation property of $\mathbf{T}$. If $u\in
\left( F_{c}(I)\right) ^{\prime }\subset \left( W_{c}^{2}(I;F)\right)
^{\prime }$, then $u^{\prime \prime },$ $\mathbf{T}u^{\prime \prime }$, $q%
\mathbf{T}u$, $\left( \mathbf{T}u\right) ^{\prime \prime }\in \left(
W_{c}^{2}(I;F)\right) ^{\prime }$ and using (\ref{secderadjoint}) we see
that 
\begin{equation*}
\left\langle \left( \mathbf{T}u\right) ^{\prime \prime }-q\mathbf{T}u,\phi
\right\rangle =\left\langle u,\mathbf{T}^{\square }(\phi ^{\prime \prime
}-q\phi )\right\rangle =\left\langle u,(\mathbf{T}^{\square }\phi )^{\prime
\prime }\right\rangle =\left\langle \mathbf{T}u^{\prime \prime },\phi
\right\rangle
\end{equation*}%
holds for all $\phi \in W_{c}^{2,p}(I)$.
\end{proof}

\bigskip

\begin{remark}
Let $q\in F_{loc}(I)$. Reasoning as in Corollary \ref{Cor_dist=clas}, we can
conclude that the equation $u^{\prime \prime }=qu$, $u\in \left(
F_{c}(I)\right) ^{\prime }$ has only weak solutions $u\in W^{2,1}(I)$. Thus,
as before, we will be able to consider standard transmutation operators.
Moreover, reasoning as in Proposition \ref{FundSDist}, we see that an
s-transmutation operator on $\left( F_{c}(I)\right) ^{\prime }$ that maps
continuously $\left( F_{c}(I)\right) ^{\prime }$ to $\left(
W_{c}^{1}(I;F)\right) ^{\prime }$ must have the form $\mathcal{T}=c_{1}%
\mathbf{T}P_{+}+c_{2}\mathbf{T}AP_{+}+c_{3}\mathbf{T}\tfrac{d}{dx}P_{-}+c_{4}%
\mathbf{T}P_{-}$ where the $c_{i}$ are constants. Moreover, $\mathcal{T}$
maps $\left( F_{c}(I)\right) ^{\prime }$ to itself continuously if and only
if $c_{3}=0$.
\end{remark}

\section{Conclusions and generalizations}

A method to compute the general formula of the s-transmutation operators
corresponding to $L=\frac{d^{2}}{dx^{2}}-q(x)$ has been presented. On the
one hand, the method is based on fairly simple formulas, which allows one to
compute a fundamental system of s-transmutation operators when one single
(one-to-one) s-transmutation operator is known in closed form. On the other
hand, with the aid of classical results we established that a particular
s-transmutation operator can be represented in the form of a Volterra
integral operator of the second kind.

Most of the results of Sections 3 and 4 can be extended to the second-order
linear differential operator $L_{1}=a_{2}(x)\frac{d^{2}}{dx^{2}}+a_{1}(x)%
\frac{d}{dx}+q(x)$ with smooth coefficients. This includes the explicit
formulas leading to the construction of new s-transmutations and then to the
general formula of the s-transmutations corresponding to $L_{1}$, once the
closed form of one single s-transmutation corresponding to $L_{1}$ is known.
If $a_{2}(x)\equiv 1$, then a closed form for such an s-transmutation can be
obtained following \cite[Chap. 1, Th. 7.1]{Lions} and if $L_{1}=\frac{d}{dx}p%
\frac{d}{dx}+q(x)$, by following \cite{KST}.

It would also be interesting to develop this method for higher order linear
differential operators. However, apart from the results of \cite{Beghar} and 
\cite{KirSingular}, there are, as far as our knowledge extends, no results
known about the existence of transmutation operators for higher orders.


\begin{thebibliography}{99}
\bibitem{Beghar} {\small Begehr H., and Gilbert R., \emph{Transformations,
Transmutations and Kernel Functions}, vols. 1--2. Pitman, London, 1992.}

\bibitem{Neto} {\small Barros-Neto J., \emph{An Introduction to the Theory
of Distributions}. Marcel Dekker, New York, 1973.}

\bibitem{Bers Book} {\small Bers L., \emph{Theory of Pseudo-Analytic
Functions}. New York University, 1952.}

\bibitem{BCK} {\small Blancarte H., Campos H. M., and Khmelnytskaya K. V.,
Spectral parameter power series method for discontinuous coefficients. \emph{%
Mathematical Methods in the Applied Sciences} 38(10), 2000--2011, 2015.}

\bibitem{Brezis} {\small Brezis H., \emph{Functional Analysis, Sobolev
Spaces and Partial Differential Equations}. Springer-Verlag, Berlin, 2011.}

\bibitem{CamposPhd} {\small Campos H. M., \emph{Bicomplex Pseudoanalytic
Functions and Applications in Constructive Methods for Solving Boundary
Value Problems}. PhD thesis, Cinvestav, Mexico City, 2012.}

\bibitem{CMK} {\small Campos H. M., Kravchenko V. V., and Mendez L. M.,
Complete families of solutions for the Dirac equation: An application of
bicomplex pseudoanalytic function theory and transmutation operators. \emph{%
Adv. Appl. Clifford Algebr.} 22(3) 577--594, 2012.}

\bibitem{CKT} {\small Campos H. M., Kravchenko V. V., and Torba S. M.,
Transmutations, L-bases and complete families of solutions of the stationary
Schr\"{o}dinger equation in the plane. \emph{Journal of Mathematical
Analysis and Applications} 389(2), 1222--1238, 2012.}

\bibitem{Carrol} {\small Carroll R., \emph{Transmutation Theory and
applications}. North-Holland, Amsterdam, 1985.}

\bibitem{Duistermaat} {\small Duistermaat, J. J., and Kolk, J. A. C., \emph{%
\ Distributions: Theory and Applications}. Birkh\"{a}user, New York, 2010.}

\bibitem{Fage} {\small Fage M. K., and Nagnibida N. I., \emph{The Problem of
the Equivalence of Ordinary Linear Differential Operators}. Nauka,
Novosibirsk, USSR, 1987 (in Russian).}

\bibitem{Grubb} {\small Grubb G., \emph{Distributions and Operators}.
Springer-Verlag, Berlin, 2009.}

\bibitem{KKTmod} {\small Khmelnytskaya K. V., Kravchenko V. V., and Torba S.
M., Modulated electromagnetic fields in inhomogeneous media, hyperbolic
pseudoanalytic functions and transmutations, submitted, available at
arXiv:1410.4873.}

\bibitem{KKTT} {\small Khmelnytskaya K. V., Kravchenko V. V., Torba S. M.,
and Tremblay S., Wave polynomials and Cauchy's problem for the Klein-Gordon
equation. \emph{J. Math. Anal. Appl.} 399, 191--212, 2013.}

\bibitem{KirSingular} {\small Kiryakova V., Transmutation method for solving
hyper-Bessel differential equations based on the Poisson--Dimovski
transformation. \emph{Fractional Calculus and Applied Analysis} 11(3), 2008.}

\bibitem{Kolmogorov} {\small Kolmogorov A. N., and Fomin S. V., \emph{%
Introductory Real Analysis}. Dover, New York, 1975.}

\bibitem{KravSPPS} {\small Kravchenko V. V., A representation for solutions
of the Sturm-Liouville equation. \emph{Complex Variables and Elliptic
Equations} 53(8), 775--789, 2008.}

\bibitem{KBook} {\small Kravchenko V. V.,\emph{\ Applied Pseudoanalytic
Function Theory}. Birkh\"{a}user, Basel, 2009.}

\bibitem{KP} {\small Kravchenko V. V., and Porter R. M., Spectral parameter
power series for Sturm--Liouville problems. \emph{Mathematical Methods in
the Applied Sciences} 33, 459--468, 2010.}

\bibitem{KST} {\small Kravchenko V. V., Morelos S., and Torba S. M.,
Liouville transformation, analytic approximation of transmutation operators
and solution of spectral problems. Submitted, available at arXiv:1412.5237.}

\bibitem{KT} {\small Kravchenko V. V., and Torba S. M., Construction of
transmutation operators and hyperbolic pseudoanalytic functions. \emph{%
Complex Analysis and Operator Theory} 9(2), 379--429, 2015.}

\bibitem{KT1} {\small Kravchenko V. V., and Torba S. M., Transmutations for
Darboux transformed operators with applications. \emph{Journal of Physics A:
Mathematical and Theoretical} 45(7), 075201, 2012.}

\bibitem{KT2} {\small Kravchenko V. V., and Torba S. M., Transmutations and
spectral parameter power series in eigenvalue problems. In: \emph{Operator
Theory, Pseudo-Differential Equations, and Mathematical Physics}, pp.
209--238, Springer-Verlag, Berlin, 2013.}

\bibitem{KTacurate} {\small Kravchenko V. V., and Torba S. M., Analytic
approximation of transmutation operators and applications to highly accurate
solution of spectral problems, \emph{J. Comput. Appl. Math.} 275, 1--26,
2015.}

\bibitem{LevitanInverse} {\small Levitan B. M.,\emph{\ Inverse
Sturm--Liouville Problems}. VSP, Zeist, 1987.}

\bibitem{Lions} {\small Lions J.-L., Op\'{e}rateurs de Delsarte et probl\`{e}%
mes mixtes. \emph{Bulletin de la S. M. de France} 84, 9--95, 1956.}

\bibitem{Marchenco} {\small Marchenko V. A., \emph{Sturm--Liouville
Operators and Applications}. Birkh\"{a}user, Basel, 1986.}

\bibitem{Treves} {\small Treves F., \emph{Topological Vector Spaces,
Distributions, and Kernels}. Academic Press, New York, 1967.}
\end{thebibliography}
\end{document}